\documentclass[12pt]{amsart}
\usepackage[utf8]{inputenc}
\usepackage[all]{xy}
\usepackage{tikz-cd,amssymb}

\newcommand{\Obj}{\text{Obj}}

\newcommand{\gr}{\operatorname{gr}}
\newcommand{\be}{{\bf 1}}

\newcommand{\CC}{\mathcal{C}}
\newcommand{\C}{\mathcal{C}}

\newcommand{\ot}{\otimes}

\DeclareMathOperator{\End}{End}

\newcommand{\D}{\mathcal{D}}
\DeclareMathOperator{\Hom}{Hom}

\newcommand{\one}{\mathbf{1}}
\newcommand{\id}{\operatorname{id} }

\numberwithin{equation}{section}

\newtheorem{theorem}{Theorem}[section]

\newtheorem{corollary}[theorem]{Corollary}

\newtheorem{lem}[theorem]{Lemma}

\newtheorem{prop}[theorem]{Proposition}
\theoremstyle{definition}

\newtheorem{remark}[theorem]{Remark}

\newtheorem{ex}[theorem]{Example}

\newtheorem{defin}[theorem]{Definition}

\begin{document}
\title[Coherence for  monoidal $G$-categories and $G$-braided categories]
{Coherence for  monoidal $G$-categories and braided $G$-crossed categories}

\author{C\'esar Galindo}
\address{Departamento de Matem\'aticas\\ Universidad de los Andes\\ Carrera 1 N.
18A - 10,  Bogot\'a, Colombia}
\email{cn.galindo1116@uniandes.edu.co}

\begin{abstract}
We prove a coherence theorem for actions of  groups on  monoidal categories. As an application we prove coherence for arbitrary braided $G$-crossed categories.
\end{abstract}
\maketitle

\section{Introduction}

Given a  group $G$ and a  monoidal category $\C$, a \textit{strict  action} of $G$ on $\C$ is a group morphisms from $G$ to $\operatorname{Aut}_\otimes^{\operatorname{Strict}}(\C)$ (the group of all strict monoidal automorphisms of $\C$). For almost all situations where symmetries of monoidal categories arise, strict actions are not sufficient, mainly because the natural notion of symmetry in monoidal category theory is not a strict monoidal automorphism. Rather, a symmetry in a monoidal category is a strong monoidal \textit{auto-equivalence}. The categorical symmetries  of a monoidal category form  a monoidal category. Since every  group $G$ defines a discrete monoidal category  $\overline{G}$, the  appropriate definition of action is  a  monoidal functor from $\overline{G}$ to $\operatorname{End}_{\ot}(\C)$, where $\operatorname{End}_{\ot}(\C)$ is the monoidal category of strong monoidal endofunctors as objects and morphisms given by monoidal natural isomorphism). For simplicity, a monoidal category with a  $G$-action will be called a monoidal $G$-category. If the action is strict we will say that the monoidal $G$-category is strict. Thus, a monoidal $G$-category  is a monoidal category $\C$ with a family of monoidal auto-equivalences $\{(g_*,\psi^g):\C\to \C\}_{g\in G}$ and a family of monoidal isomorphisms $\{\phi(g,h):(gh)_*\to g_*h_*\}_{g,h\in G}$ satisfying certain coherence axioms (see Section \ref{sectio defini G-category}).

The distinction between strict monoidal $G$-categories and  general monoidal $G$-categories is analogous to the relation between strict monoidal categories and  general monoidal categories. In monoidal category theory, the MacLane coherence theorem says that for two expressions $S_1, S_2$ obtained from $X_1\otimes X_2\ot \dots \ot X_n$ by inserting $\one$’s and brackets, any pair of isomorphisma $\Phi:S_1\to S_2$, composed of the associativity and unit constraints and their inverses, are equal. A similar presentation of coherence for arbitrary monoidal $G$-categories can be stated. However, an equivalent statement and convenient  way of expressing the coherence theorem for monoidal $G$-categories is the following: 

\begin{theorem}[Coherence for monoidal $G$-categories]\label{coherencia}
Let $G$ be a  group. Every monoidal $G$-category is equivalent to a strict monoidal $G$-category.
\end{theorem}
In essence, Theorem \ref{coherencia} says that in order to prove a general statement for monoidal $G$-categories,  we may assume without loss of generality to assume that we are working with strict $G$-categories. 

The main result of this paper is to prove Theorem \ref{coherencia}. For this, given a   group $G$ and a monoidal $G$-category $\C$, we construct a  strict monoidal $G$-category $\C(G)$ and a adjoint monoidal $G$-equivalence  $\mathcal{F}:\C\to \C(G)$. In fact, we construct a strict left adjoint 2-functor to the forgetful 2-functor from the 2-category of strict monoidal $G$-categories to the 2-category of monoidal $G$-categories.  

Braided $G$-crossed categories  are  interesting because they have applications to mathematical physics \cite{barkeshli2014symmetry,cui2015gauging,lan2016modular} and low-dimensional topology \cite{MR2674592,MR2959440,MR3195545}. As an application of Theorem \ref{coherencia}, we  prove coherence theorems for general $G$-crossed categories and braided $G$-crossed categories. This coherence theorem generalizes the  M\"{u}ger's coherence theorem for  braided $G$-crossed fusion categories over algebraically closed field of characteristic zero and $G$  finite,  \cite[Appendix 5, Theorem 4.3]{MR2674592}. M\"{u}ger's coherence theorem is obtained as corollary of depth and difficult to prove characterization of braided $G$-crossed fusion categories. The inconveniences of  \cite[Appendix 5, Theorem 4.3]{MR2674592} are that the construction of the strictification is not direct, $G$ must be finite and the conditions on the braided $G$-crossed category is very restrictive. In \emph{loc. cit.} Michael M\"{u}ger asked for a proof of the coherence in a more direct way, extending its domain of validity. Theorem \ref{coherence G-braid} has no restriction on $G$ or the underlying category $\C$ and the constructions of the strictification $\C(G)$ is very explicit. Thus,  Theorem \ref{coherence G-braid} answers M\"{u}ger's question. 

The paper is organized as follows. Section 2 contains preliminaries and notations. In Section 3,
we define the 2-category of monoidal $G$-categories. Section 4 contains the proof of coherence theorem for monoidal $G$-categories. In Sec. 5, we apply the main result to crossed $G$-categories and braided $G$-crossed categories.

\smallbreak\subsubsection*{Acknowledgements} The author is grateful to Paul Bressler for  useful discussions. This research was partially supported by the FAPA funds from Vicerrector\'{i}a de Investigaciones de la Universidad de los Andes.
\section{preliminaries and notations}

Let $\C$ be a category. We denote by $\Obj(\C)$ the class of objects of $\C$ and by $\Hom_\C(X, Y )$ the set of morphisms
in $\C$ from an object $X$ to an object $Y$. Also, by abuse of notation  $X \in \C $ means that $X$ is an object of $\C$.

The symbols  $\CC$ and $\D$ will denote  monoidal categories with unit objects $\one_\CC$ and $\one_\D$ respectively. If no confusion arise, we will denote the unit object of a monoidal category just by $\one$. In order to simplify computations and statements, by monoidal category we will mean a strict monoidal category. This is justified by  the  coherence theorem of S. MacLane  \cite{CWM,cohemac}.

\subsection{Monoidal  functors}\label{sect-monofunctor}
Let $\CC$ and $\D$ be  monoidal categories. A  \emph{monoidal
functor} from $\CC$ to $\D$ is a triple $(F,F_2,F_0)$, where $F:
\CC \to \D$ is a functor, $$ F_2=\{F_2(X,Y) : F(X \otimes Y)
\to F(X) \otimes F(Y)\}_{X,Y \in \CC} $$ are  natural isomorphisms, $F_0:F(\one)\to \one $ is an isomorphism, such that the diagrams
\begin{equation}\label{ax monoidal 1}
\begin{tikzcd}
F(X\ot Y\ot Z) \arrow{dd}{ F_2(X,Y\ot Z)}  \arrow{rrr}{F_2(X\ot Y,Z)} &&& F(X\otimes Y)\ot  F(Z) \arrow{dd}{F_2(X,Y)\ot F(Z)}\\ &&&\\
F(X)\ot F(Y\ot Z) \arrow{rrr}{\id_{F(X)}\ot F_2(Y,Z)}&&& F(X)\ot F( Y)\ot F(Z)
\end{tikzcd}
\end{equation}

\begin{equation} \label{ax monoidal 2}
\begin{tikzcd}
F(X) \arrow{d}{F_2(\one,X)} \arrow{rrr}{F_2(X,\one)} \arrow{rrrd}{\id_{F(X)}} &&& F(X)\ot F(\one) \arrow{d}{\id_{F(X)}\ot F_0}\\
F(\one)\ot F(X)\arrow{rrr}{F_0\ot \id_{F(X)}} &&& F(X)
\end{tikzcd}
\end{equation}
commute for all objects $X,Y,Z \in \CC$.

A monoidal functor $(F,F_2,F_0)$ is  called \emph{unital}
if $F_0$ is the identity morphisms and \emph{strict}
if $F_2$ and $F_0$ are identity morphisms.

\begin{remark}
A monoidal functor $(F,F_2,F_0)$ is  unital if and only if $$F_2(X,\one)=F_2(\one,X)=\id_X$$ for all $X\in \Obj(X)$.
\end{remark}

Let $F,G:\CC \to \D$ be two monoidal
functors. A natural transformation $\varphi=\{\varphi_X : F(X) \to
G(X)\}_{X \in \CC}$ from $F$ to $G$ is \emph{monoidal} if the diagrams
\begin{equation}\label{nattrans1}
\begin{tikzcd}
F(\one) \arrow{rd}{F_0}\arrow{rr}{\varphi_\one}&& G(\one) \arrow{ld}{G_0}\\
& \one&
\end{tikzcd}
\end{equation}
and
\begin{equation}\label{nattrans2}
\begin{tikzcd}
F(X\ot Y)\arrow{rr}{F_2(X,Y)} \arrow{d}{\varphi_{X\ot Y}}&& F(X)\ot F( Y) \arrow{d}{\varphi_{X}\ot\varphi_{ Y}}\\
G(X\ot Y)\arrow{rr}{G_2(X,Y)} && G(X)\ot G(Y)
\end{tikzcd}
\end{equation}
commute for all objects $X,Y$ of $\C$.

\begin{remark}
If $F,G:\CC \to \D$ are two unital monoidal
functors then condition of \eqref{nattrans1} is just $\varphi_\one=\id_\one$.
\end{remark}

If $F: \CC \to \D$ and $G: \D \to \mathcal{E}$ are two monoidal
functors, then their composition $GF:
\CC \to \mathcal{E}$
is a monoidal functor with $(GF)_0=F(G_0)F_0$ and
$$ (GF)_2=\{ G_2(F(X),F(Y)) G(F_2(X,Y))\}_{X,Y \in \CC}.
$$

A monoidal functor $(F,F_2,F_0):\CC\to \D'$ is called \emph{a monoidal equivalence} if the functor $F$ is a equivalence of categories. If $(F,F_2,F_0)$ is a monoidal equivalence, the adjoint functor of $F$ has a canonical monoidal structure, and the unit and counit of the adjointion are monoidal natural isomorphisms, see
\cite[Proposition 4.4.2]{Saavedra}.

\section{$G$-categories}

Given monoidal categories $\C$ and $\D$,  we will denote by
$\operatorname{Fun}_\otimes(\C,\D)$ the category of  monoidal functors from $\C$ to $\D$. The objects are monoidal functors
from  $\C$ to $\D$ and morphisms are monoidal natural isomorphisms. If $\C=\D$, we denote $\operatorname{Fun}_\otimes(\C,\D)$ just by $\operatorname{End}_\otimes(\C)$. The category $\operatorname{End}_\otimes(\C)$ is a strict monoidal category with tensor product given by composition of monoidal functors and unit object the identity endofunctor $\operatorname{Id}_\C.$

Analogously, given monoidal categories $\C$ and $\D$, we define $\operatorname{Fun}_\otimes^u(\C,\D)$ ( respectively $\operatorname{End}_\otimes^u(\C)$) as the full subcategory of  unital monoidal functors from $\C$ to $\D$ (respectively the monoidal category of unital monoidal endofunctors of $\C$).

\subsection{Definition of monoidal $G$-categories.}\label{sectio defini G-category}
Let $G$ be a  group (or a monoid). We will denote by $\overline{G}$ the discrete monoidal category with $\Obj(\overline{G})=G$ and monoidal structure defined by the multiplication of $G$.

A monoidal $G$-category is pair $(\psi,\C)$, where $\C$ is a monoidal category and $\psi:\overline{G}\to \End_{\otimes}(\C)$ is a monoidal functor. Two monoidal $G$-categories  $(\psi,\C)$ and $(\psi',\C)$ are called \textit{strongly} monoidally $G$-equivalent if  $\psi$ and $\psi'$ are monoidal equivalent.

 A monoidal $G$-category is called unital if $\psi$ is a unital monoidal functor from $\overline{G}$ to $\End_{\otimes}^u(\C)$.
\begin{prop}\label{G-cat son unitarios}
Every monoidal $G$-category $(\psi,\C)$ is strongly $G$-equivalent to a unital monoidal $G$-category $(\psi',\C)$.   
\end{prop}
\begin{proof}
It is clear that the proposition follows from the following statement: Let $\C$ and $\D$ be monoidal categories. Then, the inclusion functors $\iota:\operatorname{Fun}_\otimes^u(\C,\D)\to \operatorname{Fun}_\otimes(\C,\D)$ and $\iota:\operatorname{End}_\otimes^u(\C)\to \operatorname{End}_\otimes(\C)$ are equivalences of categories and of monoidal categories, respectively.

Recall that a functor is an equivalence if and only if it is essentially surjective and fully faithful. It is clear that the functors $\iota$ are fully faithful. We only need to see that they are surjective.

By \eqref{ax monoidal 2}, $F_2(X,\be_\C)= \id_{F(X)}\otimes F_0^{-1}$,
$F_2(\be_\C,Y)=F_0^{-1}\otimes \id_{F(Y)}$. Thus, defining  $$F'(X)= \left\{
            \begin{array}{ll}
              F(X), & \hbox{$X\neq \be_\C$; } \\
              \be_\D, & \hbox{$X= \be_\C$.}
            \end{array}
          \right.$$

   $$F_2'(X,Y)=\left\{
     \begin{array}{ll}
       F_2(X,Y), & \hbox{$X\neq \be_\C, Y\neq \be_\C$;} \\
       \id_{F(X)}, & \hbox{if $Y=\be_\C$;}\\
       \id_{F(Y)}, & \hbox{if $X=\be_\C$.}
     \end{array}
   \right.$$
we have a unital monoidal functor $(F',F_2'):\C\to \D$. Finally,  a monoidal natural isomorphisms $\sigma:F\to F'$ is defined by

$$\sigma(X)=\left\{
     \begin{array}{ll}
       \id_{F(X)}, & \hbox{$X\neq \be_\C$;} \\
       F_0, & \hbox{if $X=\be_\C$.}
     \end{array}
   \right.$$
In fact, $\sigma$ is a natural isomorphism since $\Hom_\D(\one_D,\one_D)$ is a commutative monoid and $hh'=h\otimes h'$ for all $h,h'\in \Hom_\D(\one_D,\one_D)$.
\end{proof}
As a consequence of Proposition \ref{G-cat son unitarios}, \textit{from now on we will consider only unital monoidal $G$-categories and we just call them monoidal $G$-categories.}

A monoidal $G$-category $\CC$ consists of the following data:
\begin{itemize}
\item functors $g_*:\C\to \C$ for each $g\in G$,
\item natural isomorphisms $\phi(g,h): (gh)_*\to g_*\circ h_*$ for each pair $g,h\in G$,
\item natural isomorphisms $\psi^g(X,Y): g_*(X\ot Y)\to g_*(X)\ot g_*(Y)$ for all $X,Y \in \Obj(\C)$,
\end{itemize}

such that 

\begin{enumerate}
\item $g_*(\one)=\one$
    \item $\psi^g(X,\one)=\psi^g(\one,X)=\id_X$,
     \item $e_*=\operatorname{Id}_\C$,
    \item $\phi(e,g)=\phi(g,e)=\operatorname{Id}_{g_*}$
\end{enumerate}
 for all $g\in G$, $X\in \Obj(\C)$ and, for all
$g,h,k \in G$ and   $X,Y,Z \in \Obj(\CC)$,   the following diagrams commute:

\begin{equation}\label{otro mas}
\begin{tikzcd}
(ghk)_*(X) \ar{rr}{\phi(gh,k)(X)}  \ar{dd}{\phi(g,hk)(X)}&& (gh)_*k_*(X) \ar{dd}{\phi(g,h)(k_*(X))}\\ \\
g_*(hk)_*(X) \ar{rr}{g_*(\phi(h,k)(X))}&& g_*h_*k_*(X)
\end{tikzcd}
\end{equation}

\begin{equation}\label{g es monoidal}
\begin{tikzcd}
g_*(X\ot Y\ot Z) \ar{dd}{\psi^g(X,Y\ot Z)} \ar{rrr}{\psi^g(X\ot Y,Z)}&&& g_*(X\ot Y)\ot g_*(Z) \ar{dd}{\psi^g(X,Y)\otimes \id_{g_*(Z)}}\\\\
g_*(X)\ot g_*(
Y\ot Z) \ar{rrr}{\id_{g_*(X)}\ot \psi^g(Y,Z)}&&& g_*(X)\ot g_*(Y)\ot g_*(Z)
\end{tikzcd}
\end{equation}

\begin{equation}\label{compatibilidad psi y phi}
\begin{tikzcd}
(gh)_*(X\ot Y) \ar{rr}{\psi^{gh}(X,Y)}  \ar{dd}{\phi(g,h)_{X\ot Y}} && (gh)_*(X)\ot (gh)_*(Y) \ar{dd}{\phi(g,h)_X\ot \phi(g,h)_Y}\\\\
g_*h_*(X\ot Y)\ar{rd}{g_*(\psi^h(X,Y))} &&g_*h_*(X)\ot g_*h_*(Y)\\
& g_*(h_*(X)\ot h_*(Y)) \ar{ru}{\psi^g(h_*(X),h_*(Y))}&
\end{tikzcd}
\end{equation}

A monoidal $G$-category is called strict if $\psi^g$ and $\phi(g,h)$ are identities for all $g,h\in G$.

Given $\C$ and $\D$ monoidal $G$-categories, a monoidal $G$-functor is a pair $(F,\gamma)$, where $F:\C\to \D$ is a  monoidal functor and $\gamma(g): g_*\circ F\to F\circ g_*$ is a family of monoidal natural isomorphisms indexed by $G$, such that $\eta(e)=\operatorname{Id}_F$ and for all $X\in \Obj(\C)$, $g,h\in G$ the diagrams

\begin{equation}\label{diagrama G-funtores}
\begin{tikzcd}
(gh)_* F(X)  \ar{rr}{\eta(gh)_X} \ar{dd}{\phi(g,h)_{F(X)}}&&  \ar{dd}{F(\phi(g,h)_{X})}  F((gh)_*(X)) \\\\
g_*h_*(F(X))  \ar{rd}{g_*(\eta(h)_{F(X)})}&& F(g_*h_*(X))\\
&g_*(F(h_*(X))) \ar{ur}{\eta(g)_{h_*(X)}}&
\end{tikzcd}
\end{equation}
commute.

We  say that $(F,\eta)$ is an equivalence of monoidal $G$-categories if the functor $F$ is an equivalence of categories. If $\C=\D$, a strongly equivalence is just an equivalence of monoidal $G$-categories of the form $(\operatorname{Id}_\C,\eta)$.

If $(F,\eta), (L,\chi):\C\to \D$ are monoidal $G$-functors, a monoidal natural transformation $\varphi:\C\to \D$ is called a monoidal natural transformation of $G$-categories if the diagrams
\begin{equation}
\begin{tikzcd}
F(g_*(X)) \ar{r}{\varphi_{g_*(X)} } & L(g_*(X)) \\\ar{u}{\eta(g)_X} 
g_*(F(X)) \ar{r}{g_*(\varphi_{X})}& g_*(L(X))\ar{u}{\chi(g)_X}
\end{tikzcd}
\end{equation}
commute for all $X\in \C$ and $g\in G$.

\subsection{Weak actions on Crossed modules}\label{Example action over pointed fusion categories}

The goal of this section is to present some examples of  monoidal $G$-categories associated to crossed modules. 

Recall that a \emph{crossed module} is a pair of groups $P$ and $H$, a left action of $P$ on $H$ (by group automorphisms)
\begin{align*}
  P\times H &\to H\\
(g,h) &\mapsto \ ^gh,  
\end{align*}
and a group homomorphism $\partial: H \to P$, such that $\partial$ is $G$-equivariant:
\[
\partial(^hg) = g \partial(h) g^{-1}
\]
and $\partial$ satisfies the so-called Peiffer identity:
\[
^{\partial(h_1)}h_2=h_1h_2h_1^{-1}
.\]
Note that $\operatorname{Im}(\partial)\subset P$ is a normal subgroup, then $\operatorname{coker}(\partial)$ is a group. 
\begin{ex}
\begin{itemize}
    \item Let $H$ be a normal subgroup of a group $P$. The group $P$ acts by conjugation on $H$ and $\partial$ given by the inclusion defines a crossed module $(H,P,\partial)$.

\item Let
\[1\to A\to H\overset{\partial}{\to}P \to 1,\]
be a central extension of groups. If $\iota:P\to H$ is a section (of sets) of $\partial$, the group $P$ acts on $H$ by $^ph=\iota(p)h\iota(h)^{-1}$. Since $A$ is central the action does not depend of the choice of $\iota$. The projection $\partial: H\to P$, defines a crossed module  $(H,P,\partial)$. 

\end{itemize}
\end{ex}

If $(H,P,\partial)$ and $(H',P',\partial')$ are crossed modules, a morphism $(\alpha,\phi):(H,P,\partial)\to (H',P',\partial')$ of a crossed modules is a commutative diagram of group morphisms 
\[
\xymatrix{
H \ar[d]_{\partial}  \ar[r]^{\alpha}& H' \ar[d]^{\partial'}\\
P \ar[r]_{\phi} & P'
}\]such that $\alpha(^gh)=\ ^{\phi(g)}\alpha(h)$, for all $h\in H, g\in P$. A morphism $(\alpha,\phi):(H,P,\partial)\to (H',P',\partial')$ is called a \emph{weak equivalence} if induces group isomorphisms \[\operatorname{ker}(\partial)\cong \operatorname{ker}(\partial'),\quad \operatorname{coker}(\partial)\cong \operatorname{coker}(\partial).\]

Given two morphisms $(\alpha,\phi), (\alpha',\phi'):(H,P,\partial)\to (H',P',\partial')$  we define a \emph{natural transformation}  $\theta:(\alpha,\phi)\Rightarrow (\alpha',\phi')$ as a map \[\theta:P\to H^{\operatorname{Im}(\phi)}=\{h\in H:\ ^gh=h, \forall g\in \operatorname{Im}(\phi)\}\] such that 
\[\phi'(g)=\partial(\theta(g))\phi(g),\quad
    \theta(g)\alpha'(h)=\alpha(h)\theta(g),\]
\[2\]
for all $g \in P, h\in H$.

Let $G$ be a group and $(H,P,\partial)$ a crossed module. A \emph{weak action} of $G$ on $(H,P,\partial)$ consists of the following data:
\begin{itemize}
    \item a morphism $(\alpha_x,\phi_x)$ for each $x\in G$,
    \item natural transformations $\theta_{x,y}: (\alpha_{xy},\phi_{xy})\to (\alpha_x\alpha_y,\phi_x\phi_y)$ for each pair $x,y\in G$,
\end{itemize}such that 
\begin{enumerate}
\item $(\alpha_e,\phi_e)= (\operatorname{Id}_H,\operatorname{Id}_G)$
\item $\theta_{e,x}=\theta_{x,e}$ are the constant function $e$.
\item $\theta_{xy,z}(g)\theta_{x,y}(\phi_z(g))=\theta_{x,yz}(g)\alpha_x(\theta_{y,z}(g))$
\end{enumerate}for all $x,y,z\in G, g\in P$

\begin{ex}
Let \[1\to H\overset{\partial}{\to} P\overset{\pi}{\to} G\to 1,\] be an exact sequence of groups and $\iota: Q\to G$ a section (of sets) of $\pi$. The group $G$ acts on  $(H,P,\partial)$ by \[\phi_g(x)=\iota(g)x\iota(g)^{-1},\quad \alpha_g(h)=\iota(g)h\iota(g)^{-1}, \]
\[\theta_{g,h}(x)=\alpha_g\circ \alpha_h(x)\alpha_{gh}(x)^{-1},\]for all $h\in H,x\in P, g,h\in G$
\end{ex}

We can build a small strict monoidal category $\C(H, P , \partial)$ (in fact, a strict categorical group) from a crossed module $(H, P , \partial)$ as follows. First we let
the set of object by $P$ and the set of arrow the semidirect product $H \rtimes P $ in which tensor product is given by the multiplication 
\[(h,g)(h',g')=(h(^gh'),gg').\]
We define source and target maps $s, t:  H\rtimes P\to P$ by:
\[s(h,g)=g,\quad t(h,g)=\ \partial(h)g,\] define the identity-assigning map 
\begin{align*}
    i:P&\to H\rtimes P\\
    g &\mapsto (1,g),
\end{align*}and define the composite of morphisms \[(h,g):g\to g',\quad (h',g'):g'\to g'',\]to be \[(hh',g):g\to g''.\]
See \cite{crossed} for more details.

Every weak action $\{\alpha_x,\phi_x,\theta_{x,y}\}_{x,y\in G}$ of a group $G$ on a crossed module $(H,P,\partial)$ defines an action of $G$ on the monoidal category $\C(H,P,\partial)$. In fact, every morphism $(\alpha_x,\phi_x)$  defines a strict monoidal functor  $F_{(\alpha_x,\phi_x)}:\C(H,P,\partial)\to \C(H,P,\partial),$  \[F_{(\alpha_x,\phi_x)}(g)=\phi_{x}(g),\quad F_{(\alpha_x,\phi_x)}(h,g)=(\alpha_x(h),\phi_x(g)),\] and every natural transformation $\theta_{x,y}$ defines a monoidal natural isomorphism $\theta_{x,y}: F_{(\alpha_{xy},\phi_{xy})}\to F_{(\alpha_x,\phi_x)}\circ F_{(\alpha_x,\phi_x)}$ by \[(\theta_{x,y}(g),\phi_{xy}(g)):\phi_{xy}(g)\to \phi_x\circ \phi_y(g).\]

\section{Coherence for monoidal $G$-categories}
Let $G$ be a  group and let $\C$ be a monoidal $G$-category. We define the category $\C(G)$ as follows:
the objects of $\C(G)$ are pairs $(L,\eta)$, where $L=\{L_g\}_{g\in G}$ is a family of objects of $\C$ and 
\begin{equation}
\eta=\{\eta_{g,h}:g_*(L_h) \to L_{gh} \}_{(g,h)\in G\times G}
\end{equation}is a family of isomorphisms such that $\eta_{e,g}=\id_{L_g}$ for all $g\in G$ and  the diagrams

\begin{equation}\label{def de L}
\begin{tikzcd}
 (gh)_*(L_k) \ar{d}{\phi(g,h)_{L_k}}  \ar{rr}{\eta_{gh,k}}  &&  L_{ghk}   \\
g_*h_*(L_k) \ar{rr}{g_*(\eta_{h,k})} && g_*(L_{hk}) \ar{u}{\eta_{g,hk}}
\end{tikzcd}
\end{equation}commute for all $g,h,k\in G$.

A morphism from $(L,\eta)$ to $(T,\chi)$ is a familiy of morphisms \[f=\{f_g:L_g\to T_g\}_{g\in G}, \] such that 
the diagrams
\begin{equation}\label{comm tensor prod flechas}
\begin{tikzcd}
g_*(L_h)  \ar{d}{g_*(f_h)}  \ar{rr}{\eta_{g,h}}  &&   L_{gh}  \ar{d}{f_{gh}} \\
g_*(T_h) \ar{rr}{\chi_{g,h}} && T_{gh}
\end{tikzcd}
\end{equation}commute for all $g,h \in G$. The composition of  $f: (L,\eta)\to (L',\eta')$ and $l:(L',\eta')\to (L'',\eta'')$ is $lf:=\{l_gf_g\}_{g\in G}$. This  composition is well defined because of  functoriality of  $g_*$.

Consider now the functor $\mathcal{F}: \C\to \C(G)$ defined as follows: For every object $X$ in $\C$, $\mathcal{F}(X)=(\overline{X},\phi_X^{-1})$, where $\overline{X}_g=g_*(X)$ and $$(\phi_X^{-1})_{g,h}:=\phi(g,h)_X^{-1}:g_*(\overline{X}_h)=g_*h_*(X)\to \overline{X}_{gh}=(gh)_*(X).$$ The pair  $(\overline{X},\phi_X^{-1})$ is an object in $\C(G)$ by the commutativity of diagram \eqref{otro mas}. Given $l:X\to Y$ a morphism in $\C$,  we define $\mathcal{F}(l):(\overline{X},\phi_X^{-1})\to (\overline{Y},\phi_Y^{-1})$ as $\mathcal{F}(l)=\{g_*(l)\}_{g\in G}$. The family of morphisms $\mathcal{F}(l)$ is a morphism in $\C(G)$ because  $\phi^{-1}(g,h):g_*h_*\to (gh)_*$ are natural isomorphisms for each pair $g,h\in G$ and $\mathcal{F}(l)\circ \mathcal{F}(f)=\mathcal{F}(l\circ f)$ because $g_*$ are functors for each $g\in G$.

\begin{prop}\label{prop equivalencia}
The functor $\mathcal{F}: \C\to \C(G)$ is an adjoint equivalence of  categories with adjoint functor $U_e:\C(G)\to \C, (L,\eta)\mapsto L_e, f\mapsto f_e$.
\end{prop}
\begin{proof}
We define the unit isomorphism of the adjoint equivalence $\eta:  \mathcal{F}\circ  U_e\to \operatorname{Id}_{\C}$ as  $$\eta_{(L,\eta)}:=\{\eta_{g,e}\}_{g\in G}.$$ The naturality of  $\eta$ follows from the diagrams \eqref{comm tensor prod flechas}, taking $h=e$. The counit isomorphism of the adjoint equivalence is the identity since $U_e \circ \mathcal{F}=\operatorname{Id}_{\C}$. The counit-unit equations follow immediately from the definitions.
\end{proof}

\begin{prop}
The category $\C(G)$ is a strict monoidal category with tensor product of objects $(L,\eta)\otimes (L',\eta')=(LL',\eta\eta')$, where   $(LL')_g:=L_g\ot L_g',$
\begin{equation}
\begin{tikzcd}
g_*(L_h\ot L_h')   \ar{rd}{\psi^g(L_h,L_h')}    \ar{rr}{(\eta\eta')_{g,h}} && L_{gh}\otimes L_{gh}'   \\
& g_*(L_h)\ot g_*(L_h') \ar{ru}{\eta_{g,h}\ot \eta'_{g,h }}  &
\end{tikzcd}
\end{equation}
tensor product of morphisms $f\otimes l=\{f_g\otimes l_g\}_{g\in G}$, and unit object  $(\overline{\one},\id_{\one})$, where $\overline{\one}_g=\one,$ for all $g\in G$.
\end{prop}
\begin{proof}
First, we will see that the tensor product is well defined.   In order to check that $(L,\eta)\otimes (L',\eta')\in \C(G)$, we need to see that  the diagrams \eqref{def de L} commute for $\eta\eta'$, that is, $$(\eta\eta')_{gh,k}=\eta_{g,hk}\circ g_*((\eta\eta')_{h.k})\phi(g,k)_{L_k\otimes L_k'},$$for all $g,h,k\in G$.

We will use the following equations:
\begin{equation}\label{naturallidad de psi}
\psi^g(L_{hk},L_{hk})\circ g_*(\eta_{h,k}\ot \eta_{h,k}')=\big ( g_*(\eta_{h,k})\ot g_*(\eta'_{h,k}) \big )\circ \psi^g(h_*(L_k),h_*(L_k'))
\end{equation} 
\begin{multline}\label{compatibilidad}
\psi^g(h_*(L_k),h_*(L_k'))\circ g_*(\psi^h(L_k,L_k'))\circ \phi(g,h)_{L_k\ot L_k'}\\ 
=\big( \phi(g,h)_{L_k}\ot \phi(g,h)_{L_k'}\big)\circ \psi^g (L_K,L_k')
\end{multline} 
for all $g,k,k\in G$. Equation \eqref{naturallidad de psi} follows because $\psi^g$ is a monoidal natural isomorphisms and equation \eqref{compatibilidad} follows from diagram \eqref{compatibilidad psi y phi}.

Thus,
\begin{multline*}
\eta_{g,hk}\circ g_*((\eta\eta')_{h.k})\circ\phi(g,k)_{L_k\otimes L_k'} =^{\tiny (1)}\ \Big[\big( \eta_{g,hk}\ot \eta'_{g,hk} \big)\circ \psi^g(L_{hk},L'_{hk}) \Big]\\
\circ g_*\Big[ \big(\eta_{h,k}\ot\eta'_{h,k}) \big)\circ \psi^h(L_K,L_k')\Big]\circ \phi(g,h)_{L_k\ot L_k'}\\
=^{\tiny (2)}\ \big( \eta_{g,hk}\ot \eta'_{g,hk} \big)\circ  \Big[ \psi^g(L_{hk},L'_{hk})\circ g_*\big(\eta_{h,k}\ot\eta'_{h,k} \big)\Big] 
\circ \Big[g_*\big(\psi^h(L_K,L_k')\big)\circ \phi(g,h)_{L_k\ot L_k'}\Big]\\
=^{\tiny (3)}\ \big( \eta_{g,hk}\ot \eta'_{g,hk} \big)\circ  \Big[\big ( g_*(\eta_{h,k})\ot g_*(\eta'_{h,k}) \big )\circ \psi^g(h_*(L_k),h_*(L_k'))\Big]  \\
\circ \Big[g_*\big(\psi^h(L_K,L_k')\big)\circ \phi(g,h)_{L_k\ot L_k'}\Big]\\
=^{\tiny (4)}\ \Big[ (\eta_{g,hk}\circ g_*(\eta_{h,k}))\ot (\eta'_{g,hk}\circ g_*(\eta'_{h,k})) \Big] \\ \circ \Big[\psi^g(h_*(L_k),h_*(L_k'))\circ g_*(\psi^h(L_k,L_k'))\circ \phi(g,h)_{L_k\ot L_k'} \Big]\\
=^{\tiny (5)}\ \Big[ (\eta_{g,hk}\circ g_*(\eta_{h,k}))\ot (\eta'_{g,hk}\circ g_*(\eta'_{h,k})) \circ \big( \phi(g,h)_{L_k}\ot \phi(g,h)_{L_k'}\big)\Big]\circ \psi^g (L_K,L_k') \\
=^{\tiny (6)}\ \Big[ (\eta_{g,hk}\circ g_*(\eta_{h,k})\circ \phi(g,h)_{L_k})\ot (\eta'_{g,hk}\circ g_*(\eta'_{h,k})\circ \phi(g,h)_{L_k}') \Big]\circ \psi^g (L_K,L_k')\\
=^{\tiny (7)}\ \big(\eta_{gh,k}\otimes \eta_{gh,k}' \big)\circ \psi^g (L_K,L_k')=^{\tiny (8)}\  (\eta\eta')_{gh,k}.
\end{multline*}
The  equality (1) from  definition, the equality (2)  by the functoriality of $g_*$, the equality (3) from equation \eqref{naturallidad de psi} , the equality (4) from bifunctoriality of $\otimes$, the equality (5) from equation \eqref{compatibilidad}, the equality (6) from bifunctoriality of $\otimes$, the equality (7) from diagram \eqref{def de L}  and the equality (8) from definition.

Let $f:(L,\eta)\to (T,\chi)$ and $l:(L',\eta')\to (T',\chi')$ be morphisms in $\CC(G)$.  In order to show that $f\otimes g:=\{ f_g\otimes l_g \}_{g\in G}$ is a morphism  $\CC(G)$ from $(L,\eta)\ot (L',\eta')$ to $(T,\chi)\ot (T',\chi')$ we need to check the equation 
\begin{equation}\label{nat}
\big( g_*(f_h)\ot g_*(l_h)\big )\circ \psi^g (L_h,L'_h)=\psi^g(T_h,T'_h)\circ g_*(f_h\otimes l_h)
\end{equation}which follows from the monoidal naturality of $\psi^g$.

Then,
\begin{multline*}
\big(f_{gh}\ot l_{gh}\big)\circ (\eta\eta')_{g,h}=^{\tiny (1)}\  \big(f_{gh}\ot l_{gh}\big)\circ (\eta_{g,h}\ot \eta_{g,h}')\circ \psi^g (L_h,L_h')\\
=^{\tiny (2)}\  \big(f_{gh}\eta_{g,h}\ot l_{gh}\eta'_{g,h}\big)\circ \psi^g (L_h,L_h')=^{\tiny (3)}\  \big(\chi_{g,h}g_*(f_h)\ot\chi_{g,h}'g_*(l_h) \big)\circ \psi^g (L_h,L_h')\\
=^{\tiny (4)}\  (\chi_{g,h}\ot \chi_{g,h}')\big( g_*(f_h)\ot g_*(l_h)\big)\circ \psi^g (L_h,L_h')\\
=^{\tiny (5)}\  (\chi_{g,h}\ot \chi_{g,h}')\circ \psi^g(T_h,T'_h)\circ g_*(f_h\otimes l_h)=^{\tiny (6)}\   (\chi\chi')_{g,h}\circ g_*(f_h\otimes l_h).
\end{multline*}
The  equality (3) follows from diagram \eqref{comm tensor prod flechas} and the equality (5) from equation \eqref{nat}.

It follows from the definition that $(\overline{\one},\id_{\one})$ is a strict unit object.

Finally, we will show that the tensor product is strictly  associative. Let $(L,\eta)$,  $(L',\eta'),$ and $(L'',\eta''),$ objects in $\C(G)$. Then,
\begin{multline*}
((\eta\eta')\eta'')_{g,h}=^{\tiny (1)}\ (\eta\eta')_{g,h}\ot \eta_{g,h}\circ \psi^g (L_h\ot L_h',L_h'')\\
=^{\tiny (2)}\  \Big(\eta_{g,h}\ot \eta_{g,h}'\ot \eta_{g,h}''\Big)\circ \Big ((\psi^g (L_h\ot L_h',L_h'')\ot\id_{g_*(L''_h)}) \circ \psi^g (L_h\ot L_h',L_h'') \Big )\\
=^{\tiny (3)}\  \Big(\eta_{g,h}\ot \eta_{g,h}'\ot \eta_{g,h}''\Big)\circ \Big ((\id_{g_*(L_h)}\ot \psi^g(L_h',L_h''))\circ \psi^g (L_h,L_h'\ot L_h'') \Big )\\
=^{\tiny (4)}\  \eta_{g,h}\ot (\eta'\eta'')_{g,h}\circ \psi^g(L_h,L_h'\ot L_h'')=(\eta(\eta'\eta')')_{g,h}
\end{multline*}
The equality (3) follows from the commutativity of diagram  \eqref{g es monoidal}. Hence, the tensor product is associative and $\C(G)$ is a strict monoidal category.
\end{proof}

The monoidal category  $\C(G)$ is a strict monoidal $G$-category with action on objects $g_*(L,\eta):=(gL,g\eta),$  where $(gL)_h:=L_{hg}$ and $(g\eta)_{x,y}:=\eta_{x,yg}$ for all $g, h,x,y\in G$ and action on morphisms $g_*(f)_h=f_{hg}$, for all $g,h\in G$.

The following theorem implies Theorem \ref{coherencia}.
\begin{theorem}\label{main theorem}
Let $\C$ be a monoidal $G$-category. The strict monoidal functor  $U_e:\C(G)\to \C$ with the natural isomorphisms $$\gamma(g)_{(L,\eta)}:=\eta_{g,e}:g_*(U_e(L,\eta))\to U_e(g_*(L,\eta))$$ is an equivalence of monoidal $G$-categories.
\end{theorem}
\begin{proof}
It follows immediately from the definitions that  $\C(G)$ is a strict monoidal $G$-category and $U_e$ is an strict monoidal functor. 

By Proposition \ref{prop equivalencia} $U_e$ is an equivalence of categories. Thus, we only need to prove that  $(U_e,\gamma)$ is a functor of monoidal $G$-categories.  The naturality of the family of isomorphism  $$\gamma(g):=\{\gamma(g)_{(L_,\eta)}:=\eta_{g,e}\}_{(L,\eta)\in \Obj(\C(G))}$$ follows from the diagrams \eqref{comm tensor prod flechas}, taking $h=e$.  Let $(L,\eta)$ and $(T,\chi)$ be objects in $\C(G)$. The monoidallity of $\gamma(g)$ is equivalent to the equation $(\eta\chi)_{g,e}=\psi^g _{L_e,T_e}\circ (\eta_{g,e}\ot \chi_{g,e})$, that follows from the definition of $(\eta\chi)_{g,h}$. Finally, the diagrams \eqref{diagrama G-funtores} for $(U_e,\eta)$ are just the diagrams \eqref{def de L}  with $k=e$.
\end{proof}
If $(F,\gamma):\C\to \D$ is a monoidal $G$-functor, we  define a monoidal $G$-functor $(F,\gamma)(G):\C(G)\to \D(G)$ as follows: if $(L,\eta )\in \C(G)$, then $F(L)=\{F(L_g)\}_{g\in G}$ and $F(\eta_{g,h})\circ \gamma_{L_h}: g_*(F(L_h))\to g_*(F(L_h))$ for all $g,h\in G$.

\begin{remark}
If the monoidal $G$-category is fusion category (or more generally a finite tensor category) over a field $k$ and every monoidal equivalence $g_*$ is $k$-linear then $U_e$ is an equivalence of fusion categories (or more generally an equivalence of  finite tensor categories). Thus, Theorem \ref{main theorem} immediately implies coherence  for fusion categories with (not necessarily finite) group actions.
\end{remark}

The following statement is equivalent to Theorem \ref{coherencia}.
\begin{corollary}
The construction $\C\mapsto \C(G)$ defines a strict left adjoint to the  forgetful 2-functor from the 2-category of strict monoidal $G$-categories to the 2-category of monoidal $G$-categories and the components of the unit are equivalences of monoidal $G$-categories.
\end{corollary}
\qed

Let $(H,P,\partial)$ be a crossed module, $G$ a group and $\{(\alpha_x,\phi_x, \theta_{x,y})\}_{x,y\in G}$ a weak action of $G$ on $(H,P,\partial)$ (see Example \ref{Example action over pointed fusion categories}). 

The strict monoidal $G$-category $\C(H,P,\partial)(G)$ is a strict monoidal category where every arrow is invertible and every object has a strict inverse. This kind of monoidal categories are called strict categorical group or strict 2-groups (see \cite{baez}). Since the notation of strict categorical group and crossed module are essentially equivalents (see \cite{crossed} for details), we can construct a new crossed module $(H(G),P(G),\partial')$ with a strict $G$ action such that is is weak equivalent to   $(H,P,\partial)$. In fact, $P(G)$ consist of the group of
objects of $\C(H,P,\partial)(G)$, $H(G)$ is  the group of morphisms $X \to  \one$ into the identity object of $\C(H,P,\partial)(G)$ 
with product: 
\[(X\overset{a}{\to} \one)\cdot (Y\overset{b}{\to} \one)= X\otimes Y\overset{a\otimes b}{\to}\one\otimes \one =\one, \] action \[^Y(X\overset{a}{\to}\one)= Y\otimes X\otimes Y^{-1} \overset{\id_Y\otimes a \otimes \id_{Y^{-1}}}{\xrightarrow{\hspace*{1.5cm}}} Y\otimes \one \ot Y^{-1}=\one.\]
and the homomorphism $\partial' : H(G) \to P(G)$ sending $X \to \one$ into $X$.

\section{Coherence for Braided $G$-crossed categories}



\subsection{$G$-crossed monoidal categories}


A $G$-graded monoidal category is  a  monoidal category $\C$ endowed with a decomposition $\C=\coprod_{g\in G} \C_g$  (coproduct of categories) such that
 \begin{itemize}
 \item $\mathbf{1} \in \C_e$,
 \item $\C_g\otimes \C_h \subset \C_{gh}$ for all $g,h\in G$.
 \end{itemize}
If $k$ is a commutative ring and $\C$ is a $k$-linear abelian category, the coproduct $\C=\coprod_{g\in G} \C_g$ is taken in the category of $k$-linear abelian categories.

\begin{defin}
A $G$-crossed monoidal categories is a $G$-graded monoidal category  with a structure of $G$-category such that $g_*(\C_h)\subset \C_{ghg^{-1}}$ for all $g,h\in G$. 

A $G$-crossed monoidal category is called $k$-linear if  $\C$ is a $k$-linear
category and the functors $g_*$ are $k$-linear for each $g\in G$.

Let $\C$ and $\D$ be $G$-crossed monoidal categories. An equivalence of monoidal $G$-categories $F:\C\to \D$ is an equivalence of  $G$-crossed monoidal categories if $F(\C_g)\subset \D_g$ for all $g\in G$.
\end{defin}

In Example \ref{Example action over pointed fusion categories} we define the notion of a weak action of a group $G$ on a crossed module $(H,P,\partial)$ and the associated monoidal $G$-category $\C(H,P,\partial)$.
\begin{lem}
Let $G$ be a group, $(H,P,\partial)$ be a crossed module and $\{(\alpha_x,\phi_x, \theta_{x,y})\}_{x,y\in G}$ a weak action of $G$ on $(H,P,\partial)$. 
\begin{enumerate}
    \item The $G$-gradings of $\C(H,P,\partial)$ are in correspondence with group homomorphisms   \[\operatorname{gr}:P\to G\] such that   $\operatorname{Im}(\partial)\subset \ker(\operatorname{gr})$.
    \item A $G$-grading given by a group homomorphism $\operatorname{gr}:P\to G$, defines a $G$-crossed monoidal structure on $\C(H,P,\partial)$ if and only if \[\operatorname{gr}(\phi_x(g))=x\operatorname{gr}(g)x^{-1},\] for all $x\in G, g\in P$.
\end{enumerate}  
\end{lem}

\begin{proof}
The group of isomorphism classes of objects of $\C(H,P,\partial)$ is $P/\operatorname{Im}(\partial)$. Since $\C(H,P,\partial)$ is a groupoid, $G$-gradings correspond with group morphisms $ P/\operatorname{Im}(\partial)\to G$. Thus, $G$-gradings correspond with groups morphism $\gr:P\to G$ such that $\operatorname{Im}(\partial)\subset \ker(\operatorname{gr})$.

The second part of the lemma  follows immediately from the definition.
\end{proof}

Let $\C$ and $\D$ be $G$-crossed monoidal categories. A monoidal $G$-functor $(F,\eta):\C\to \D$ is a functor of $G$-crossed monoidal categories if $F(\C_g)\subset \D_g$ for all $g\in G.$ We  say that $(F,\eta)$ is an equivalence of $G$-crossed monoidal categories if $F$ is an equivalence of categories.

A $G$-crossed monoidal categories is called strict if it is strict as a monoidal $G$-category.

\begin{corollary}[Coherence for $G$-crossed monoidal categories]\label{Coherence for $G$-crossed monoidal categories}
Let $G$ be a  group. If $\C$ is a $G$-crossed monoidal category, $\C(G)$ is a strict $G$-crossed monoidal category equivalent to $\C$. If $\C$ is $k$-linear, $\C(G)$ is $k$-linear and equivalent to $\C$ as $k$-linear categories.
\end{corollary}
\begin{proof}
The strict monoidal $G$-category $\C(G)$ is a $G$-crossed monoidal category with $\C(G)_g=\{(L,\eta): L_e\in \C_g\}$ for all $g\in G$.

If each $g_*$ is $k$-linear, the strict monoidal $G$-category $\C(G)$ is $k$-linear  and the equivalence of categories  $U_e:\C(G)\to \C$ is a $k$-linear equivalence such that $U_e(\C(G)_g)\subset \C_g$ for all $g\in G$.  Thus, the corollary follows from Theorem \ref{main theorem}
\end{proof}
\subsection{Braided $G$-crossed categories}
\begin{defin}
Let $\C$ be a $G$-crossed monoidal category. A $G$-braiding is a family of  isomorphisms \[c:=\{c_{X,Y}:X\ot Y\to g_*(Y)\ot X\}_{Y\in \C, X\in \C_{g}, g\in G}\] natural in $X$ and $Y$, such that the diagrams

\begin{equation}\label{axiom 1 trenza}
\begin{tikzcd}
g_*(X\ot Z) \ar{dd}{\psi^g(X,Z)} \ar{rrrr}{g_*(c_{X,Z})} &&&& g_*(h_*(Z)\ot X)\ar{dd}{\psi^g(h_*Z, h_*(X))}\\\\   
g_*(X)\ot g_*(Z) \ar{dd}{c_{g_*(X),g_*(Z)}} &&&& g_*h_*(Z)\ot g_*(X) \\\\  
(ghg^{-1})_*g_*(Z)\ot g_*(X) \ar{rrrr}{\phi(ghg^{-1},g)_Z^{-1}\ot \id_{g_*(Z)} }  &&&& (gh)_*(Z)\ot g_*(X) \ar{uu}{\phi(g,h)_X\ot \id_{g_*(X)}}
\end{tikzcd}
\end{equation}
commute  for all $X\in \C_h, Z\in \C, g,h\in G$, the diagrams
\begin{equation}\label{trenzas G}
\begin{tikzcd}
X \ot Y\ot Z \ar{rr}{c_{X,Y\ot Z}} \ar{d}{c_{X,Y}\ot\id_{Z}} &&   g_*(Y\ot Z)\ot X \ar{d}{\psi^g(Y,Z)\ot \id_X} \\
g_*(Y)\ot X\ot Z \ar{rr}{\id_{g_*(Y)}\ot c_{X,Z}}&&  g_*(Y)\ot g_*(Z) \ot X
\end{tikzcd}
\end{equation}
commute for all $X\in \C_g, Y,Z\in \C$ and the diagrams 

\begin{equation}\label{trenzas 2}
\begin{tikzcd}
X \ot Y\ot Z \ar{rr}{c_{X\ot Y, Z}} \ar{d}{\id_X\ot c_{Y,Z}} &&   (gh)_*(Z)\ot X\ot Y  \ar{d}{\phi(g,h)_Z\ot \id_{X\ot Y}} \\
X\ot h_*(Z)\ot Y\ar{rr}{c_{X,h_*(Z)}\ot \id_Y}&&  g_*h_*(Z)\ot X \ot Y
\end{tikzcd}
\end{equation}
comute for all $X\in \C_g, Y\in \C_h, Z\in \C, g,h\in G$.
\end{defin}

A braided $G$-crossed category is a $G$-crossed category with a $G$-braiding.
A braided $G$-crossed monoidal category $(\C,c)$ is called strict if $\C$ is a strict monoidal $G$-category. In a strict braided $G$-crossed monoidal category we have that 
\begin{itemize}
    \item $g_*(c_{X,Z})=c_{g_*(X),g_*(Z)}$ 
    \item $c_{X,Y\ot Z}=(\id_Y\ot c_{X,Z})\circ (c_{X,Y}\ot \id_{Z})$
    \item $c_{X\ot Y,Z}=(c_{X,h_*(Z)}\ot \id_{Y}) \circ (\id_X\ot c_{Y,Z})$
\end{itemize}for all $X\in \C$, $Y\in \C_g, Z\in \C_h,$ $g,h\in G$.
Let $\C$ and $\D$ be braided $G$-crossed categories. 

\begin{ex}
\begin{enumerate}
    \item Let $(H,P,\partial)$ be a crossed module. Since $^{\partial(h)}h'h=hh'$ for all $h,h'\in H$, the discrete monoidal category $\overline{H}$ is a strict braided $P$-crossed category.
    \item Let $(\mathcal{B},c)$ be a braided monoidal category and $G$ a group with an action on $\mathcal{B}$ by braided autoequivalences. Then $\mathcal{B}$, with all objects graded only by $e\in G$ and $G$-braiding $c$ is a braided $G$-crossed monoidal category.
    
\end{enumerate}
\end{ex}

A functor of $G$-crossed monoidal categories  $(F,\eta):\C\to \D$ is a functor of braided $G$-crossed monoidal categories if for all $X\in \C_g, Y\in \C, g\in G$ the diagrams
\begin{equation}\label{braid G-functor}
\begin{tikzcd}
F(X\ot Y)  \ar{rr}{F(c_{X,Y})} \ar{dd}{F_2(X,Y)}&&  \ar{dd}{F_2(g_*(Y),X)} F(g_*(Y)\ot X) \\\\
F(X)\ot F(Y)  \ar{rd}{c_{F(X),F(Y)}}&& F(g_*(Y))\ot F(X)\\
& g_*(F(Y))\ot F(X)\ar{ur}{ \eta(g)_Y\ot \id_{F(X)}}&
\end{tikzcd}
\end{equation}
commute.

We  say that $(F,\eta)$ is an equivalence of braided $G$-crossed monoidal categories if $F$ is an equivalence of categories.

\begin{theorem}[Coherence for braided $G$-crossed  categories]\label{coherence G-braid}
Let $G$ be a  group. Every braided $G$-crossed  category is equivalent to a strict braided $G$-crossed category.
\end{theorem}
\begin{proof}
Using the adjoint equivalence of monoidal $G$-crossed categories, $(U_e,\mathcal{F})$, we will transport the $G$-braiding of $\C$ to a $G$-braiding on $\C(G)$.
The $G$-braiding on  $\C(G)$ is defined by the commutativity of the diagram
\[
\xymatrixrowsep{0.6in}
\xymatrixcolsep{0.6in}
\xymatrix{
L_h\ot T_h    \ar[rr]^{\tilde{c}_{L_h,T_h}} && T_{hg}\ot L_h\\
h_*(L_e)\ot h_*(T_e)\ar[u]^{\eta_{h,e}\ot \chi_{h,e}} &&h_*(T_g)\ot h_*(L_e) \ar[u]_{\eta_{h,g}\ot \chi_{h,e}}\\
h_*(L_e\ot T_e) \ar[u]^{\psi^h(L_e,T_e)} \ar[rd]_{h_*(c_{L_e,T_e})} && h_*(T_g\ot L_e)\ar[u]_{\psi^h(L_g,T_e)} \\
&h_*(g_*(T_e)\ot L_e) \ar[ru]_{h_*(\eta_{g,e}\ot \id_{L_e})} &
}
\] 
where $(L,\eta)\in \C(G)_g, (T,\chi)\in \C(G)$, $g\in G$. 
The monoidal $G$-functor  $(U_e,\gamma):\C(G)\to \C$  is a functor of braided $G$-categories. Thus,  $(U_e,\gamma)$ is an equivalence of braided $G$-crossed categories.
\end{proof}
\begin{remark}
Using the same ideas in the proof of Theorem \ref{braid G-functor}, without significant changes a coherence theorem for $G$-ribbon crossed categories can be proved.
\end{remark}

\begin{ex}
Let $(H,P,\partial)$ be a crossed module, $G$ be a group,  $\{(\alpha_x,\phi_x, \theta_{x,y})\}_{x,y\in G}$ be a weak $G$-action and $\gr: P\to G$ be a group homomorphism such that \[\operatorname{Im}(\gr)\subset \operatorname{ker}(\partial), \quad \operatorname{gr}(\phi_x(g))=x\operatorname{gr}(g)x^{-1},\] for all $x\in G, g\in P$. 

A $G$-braiding is a map \[\{-,-\}:P\times P\to H\] satisfying the
following axioms:
\begin{enumerate}
\item[(a)] $\partial(\{x,y\})xy=\phi_{\gr(x)}(y)x,$
    \item[(b)] $\{x,y\}^{\phi_{\gr(x)}(y)}h=h\{\partial(h)x,y\}$,
   \item[(c)] $^xh\{x,\partial(h)y\}=\{x,y\}\alpha_{\gr(x)}(h)$,
   \item[(d)] $\{\phi_g(x),\phi_g(y)\}=\alpha_g(\{x,y\})\theta_{g,\gr(x)}^{-1}\theta_{g\gr(x)g^{-1},g}$,
   \item[(e)] $\{x,yz\}=\{x,y\}^{\phi_{\gr(x)}(y)}\{x,z\}$,
   \item[(f)] $\{xy,z\}\theta_{\gr(x),\gr(y)}(z)=\ ^x\{y,z\}\{x,\phi_{\gr(y)}(z)\}$,
\end{enumerate}for all $x,y,z\in P,$   $g\in G$.


If $G$ is a trivial group, we obtain the notion of braiding of crossed module, see \cite{braidcrossed}.

If the $G$-action is strict, that is,  $\theta_{x,y}(g)=e$, for all $x,y\in G, g\in P$, then we obtain the notion of 2-crossed module of Conduch\'e \cite{2-crossed}. 

Every $G$-braiding in a crossed module $(P,H,\partial)$ with a weak action of $G$, induces a $G$-braiding \[c_{x,y}:=(\{x,y\},xy):xy\to \phi_{\gr(x)}(y)x, \quad x,y\in P\]of $\C(P,H,\partial)$. In fact, (a) says that the target of $(\{x,y\},xy)$ is $\phi_{\gr(x)}(y)x$. Condition (b) and (d) are naturality of $c_{x,y}$. Commutativity of diagrams \eqref{axiom 1 trenza}, \eqref{trenzas G} and \eqref{trenzas 2} are equivalent to axioms (d), (e) and (f), respectively. 
\end{ex}
Recall that a categorical groups is a rigid monoidal groupoid, (see \cite{baez} for more details). Every categorical groups is equivalent to a strict categorical groups, that is, to a categorical groups where every object has a strict inverse respect to the tensor product.

Applying Theorem  \ref{coherence G-braid} we have that every braided $G$-crossed categorical groups is equivalent to a strict braided $G$-crossed strict categorical groups $\C$. Since every strict braided $G$-crossed strict categorical group defines a 2-crossed module, (see \cite[Example 2.5 (ii)]{Pilar}), associated to every braided $G$-crossed categorical group there is a 2-crossed module.


\begin{thebibliography}{10}

\bibitem{baez}
J.~C. Baez and A.~D. Lauda.
\newblock Higher-dimensional algebra. {V}. 2-groups.
\newblock {\em Theory Appl. Categ.}, 12:423--491, 2004.

\bibitem{barkeshli2014symmetry}
M.~Barkeshli, P.~Bonderson, M.~Cheng, and Z.~Wang.
\newblock Symmetry, defects, and gauging of topological phases.
\newblock {\em arXiv preprint arXiv:1410.4540}, 2014.

\bibitem{braidcrossed}
R.~Brown and N.~D. Gilbert.
\newblock Algebraic models of {$3$}-types and automorphism structures for
  crossed modules.
\newblock {\em Proc. London Math. Soc. (3)}, 59(1):51--73, 1989.

\bibitem{Pilar}
P.~Carrasco and J.~Mart\'\i~nez Moreno.
\newblock Categorical {$G$}-crossed modules and 2-fold extensions.
\newblock {\em J. Pure Appl. Algebra}, 163(3):235--257, 2001.

\bibitem{2-crossed}
D.~Conduch\'e.
\newblock Modules crois\'es g\'en\'eralis\'es de longueur {$2$}.
\newblock In {\em Proceedings of the {L}uminy conference on algebraic
  {$K$}-theory ({L}uminy, 1983)}, volume~34, pages 155--178, 1984.

\bibitem{cui2015gauging}
S.~X. Cui, C.~Galindo, J.~Y. Plavnik, and Z.~Wang.
\newblock On gauging symmetry of modular categories.
\newblock {\em arXiv preprint arXiv:1510.03475. To appear in Communications in
  Mathematical Physics}, 2015.

\bibitem{crossed}
M.~Forrester-Barker.
\newblock Group objects and internal categories.
\newblock {\em arXiv preprint math/0212065}, 2002.

\bibitem{lan2016modular}
T.~Lan, L.~Kong, and X.-G. Wen.
\newblock Modular extensions of unitary braided fusion categories and 2+ 1d
  topological/spt orders with symmetries.
\newblock {\em arXiv preprint arXiv:1602.05936}, 2016.

\bibitem{CWM}
S.~Mac~Lane.
\newblock {\em Categories for the working mathematician}, volume~5.
\newblock Springer Science \& Business Media, 1978.

\bibitem{cohemac}
S.~MacLane.
\newblock Natural associativity and commutativity.
\newblock {\em Rice Institute Pamphlet-Rice University Studies}, 49(4), 1963.

\bibitem{Saavedra}
N.~Saavedra~Rivano.
\newblock {\em Cat\'egories {T}annakiennes}.
\newblock Lecture Notes in Mathematics, Vol. 265. Springer-Verlag, Berlin-New
  York, 1972.

\bibitem{MR2674592}
V.~Turaev.
\newblock {\em Homotopy quantum field theory}, volume~10 of {\em EMS Tracts in
  Mathematics}.
\newblock European Mathematical Society (EMS), Z\"urich, 2010.
\newblock Appendix 5 by Michael M{\"u}ger and Appendices 6 and 7 by Alexis
  Virelizier.

\bibitem{MR2959440}
V.~Turaev and A.~Virelizier.
\newblock On 3-dimensional homotopy quantum field theory, {I}.
\newblock {\em Internat. J. Math.}, 23(9):1250094, 28, 2012.

\bibitem{MR3195545}
V.~Turaev and A.~Virelizier.
\newblock On 3-dimensional homotopy quantum field theory {II}: {T}he surgery
  approach.
\newblock {\em Internat. J. Math.}, 25(4):1450027, 66, 2014.

\end{thebibliography}
\end{document}